\documentclass{article}
\usepackage{graphicx}
\usepackage{cite}
\IfFileExists{authblk.sty}{\usepackage{authblk}}{}
\usepackage{amsthm}
\usepackage{amsmath}
\usepackage{amssymb}
\usepackage[left=3cm, right=3cm,top=1.0in, bottom=1.0in]{geometry}
\usepackage[colorlinks=true]{hyperref}

\date{}
\title{Graph-Compatible Power Concavity in Weakly Coupled Elliptic Systems}
\author{Jiahuan Li, Hanpeng Zhou}

\newcommand{\keywords}[1]{\par\quad\textbf{Keywords:} #1}

\begin{document}
\maketitle

\newtheorem{theorem}{Theorem}[section]
\newtheorem{definition}[theorem]{Definition}
\newtheorem{lemma}[theorem]{Lemma}
\newtheorem{corollary}[theorem]{Corollary}
\newtheorem{example}[theorem]{Example}
\newtheorem{proposition}[theorem]{Proposition}
\newtheorem{conjecture}[theorem]{Conjecture}
\newtheorem{remark}[theorem]{Remark}
\newtheorem{assumption}[theorem]{Assumption}

\begin{abstract}

For multicomponent weakly coupled elliptic systems, we introduce graph-compatible powers constrained by the support of the coupling matrix and the nonlinear exponents. We establish concavity of the associated transformed components through a concave-envelope method and a weighted viscosity comparison principle. We further establish a componentwise constant-rank theorem for the transformed Hessians, which yields strict power concavity under additional assumptions.

\end{abstract}

\keywords{Power concavity; weakly coupled elliptic systems; graph-compatible
powers; concave envelope; viscosity comparison; constant rank theorem;
strict concavity.}

\section{Introduction}

Concavity properties of solutions to nonlinear elliptic and
parabolic equations have played an important role in the qualitative theory
of partial differential equations. They provide information on the geometry of
level sets, critical points, Hessian structures, and geometric inequalities of
Brunn--Minkowski and Pr\'ekopa--Leindler type.
In particular, concavity properties of solutions to nonlinear elliptic
equations have become an important tool for understanding the interaction
between the analytic structure of equations and the geometry of their
solutions. Classical results concerning convexity of the level sets of Green functions go
back to Carath\'eodory (see \cite{Caratheodory1920}) and Gabriel
\cite{Gabriel1957}.
Makar-Limanov proved the power-concavity property for the torsion function
in planar convex domains \cite{Limanov1971}.
Brascamp and Lieb established the log-concavity of the first Dirichlet
eigenfunction of the Laplacian and the preservation of log-concavity of
initial data under the heat flow \cite{BrascampLieb1976}.

Two complementary approaches have emerged in the study of concavity of
solutions. In the macroscopic approach, Korevaar developed the concavity
maximum principle for nonlinear elliptic and parabolic boundary value
problems \cite{Korevaar1983}. Kennington then developed the
power-concavity theory for nonlinear elliptic boundary value problems, while
Kawohl refined Korevaar's concavity method and related uniqueness issues
\cite{Kawohl1986,Kennington1985}. The viscosity framework for concavity
maximum principles was further developed by Alvarez, Lasry, and Lions
\cite{AlvarezLasryLions1997}, and Salani extended concave-envelope techniques
to nonlinear elliptic and Hessian-type equations and related convexity
inequalities \cite{Salani2012}.

The complementary microscopic approach is based on constant-rank methods. Caffarelli and Friedman established a foundational constant-rank theorem for the Hessians of solutions to semilinear elliptic equations, which became a key tool for obtaining strict convexity \cite{CaffarelliFriedman1985}. Korevaar--Lewis subsequently
established constant-rank properties for convex solutions of elliptic
equations in higher dimensions \cite{KorevaarLewis1987}. Later, Guan and Ma
extended constant-rank techniques to geometric fully nonlinear elliptic
equations \cite{GuanMa2003}. Caffarelli, Guan, and Ma further developed this
theory for fully nonlinear elliptic equations and established structural
conditions ensuring constant-rank properties
\cite{CaffarelliGuanMa2007}. Bian and Guan later established
constant-rank theorems for more general fully nonlinear elliptic equations
\cite{BianGuan2009,BianGuan2010}.

For weakly coupled systems, Ishige, Nakagawa, and Salani developed a
concave-envelope method and obtained power-concavity results for the
two-component cross-coupled Lane--Emden system
\[
-\Delta u=\lambda_1v^\alpha,\qquad
-\Delta v=\lambda_2u^\beta,
\]
and for its parabolic counterpart \cite{IshigeNakagawaSalani2016}. In that
setting the two transformation powers are linked by a pair of homogeneity
relations. For a general directed coupling graph, however, one component may
depend on several neighboring components, and all active terms in the same
equation must acquire the same homogeneity after transformation. This leads
to the edgewise compatibility conditions introduced below. To the best of our
knowledge, a corresponding theory for general multicomponent systems with
nontrivial coupling graphs has not previously been developed.

In this paper, we study the multicomponent weakly coupled elliptic system
\begin{equation}
\begin{cases}
-\Delta u_i=
\displaystyle\sum_{j=1}^{m}a_{ij}u_j^{\alpha_j},
&x\in\Omega,\\[2mm]
u_i>0,
&x\in\Omega,\\
u_i=0,
&x\in\partial\Omega,
\end{cases}
\qquad i=1,\ldots,m,
\label{eq:system}
\end{equation}
where $\Omega\subset\mathbb R^n$ is bounded and convex,
$\alpha_j>0$, $a_{ij}\geq0$, and every row of $A=(a_{ij})$ contains a positive entry.
We associate with $A$ the directed graph
\[
i\longrightarrow j\Longleftrightarrow a_{ij}>0 .
\]
An edge $i\to j$, together with the corresponding term and exponent, will be
called active in the $i$-th equation.

In the multicomponent setting, the power transformations cannot be chosen
independently but must be compatible with the coupling graph. We introduce
graph-compatible powers $p_i>0$ satisfying
\begin{equation}
\frac1{p_i}=2+\frac{\alpha_j}{p_j},
\qquad a_{ij}>0,
\label{eq:compatibility}
\end{equation}
together with
\begin{equation}
0<\frac{\alpha_j}{p_j}\leq1,
\qquad a_{ij}>0.
\label{eq:small-coupling}
\end{equation}
These relations identify the natural transformations $U_i=u_i^{p_i}$. The results below extend the power-concavity theory from two-component systems to general multicomponent weakly coupled elliptic systems.

A second contribution concerns the constant-rank property of the Hessians of
the transformed components. Although constant-rank theorems have been
extensively developed for scalar nonlinear elliptic equations, to the best of
our knowledge no such result is available for multicomponent weakly coupled
elliptic systems of the present type. We establish a componentwise
constant-rank theorem and derive strict power concavity under additional
conditions.

We now state the main results of this paper.

\begin{theorem}\label{thm:existence-power}
Let $\Omega\subset\mathbb R^n$ be a bounded convex domain with $C^{1,1}$
boundary. Assume that each row of $A=(a_{ij})$ has at least one positive
entry and that there exist $p_i>0$ satisfying
\eqref{eq:compatibility}--\eqref{eq:small-coupling}. Then
\eqref{eq:system} admits a unique positive classical solution
\[
u=(u_1,\ldots,u_m),\qquad
u_i\in C^\infty(\Omega)\cap C^{1,\gamma}(\overline\Omega)
\quad\text{for every }0<\gamma<1.
\]
Moreover, $U_i:=u_i^{p_i}$ is concave in $\Omega$ for every $i$. If, in addition, $\partial\Omega\in C^{2,\sigma}$ for some
$0<\sigma<1$, then
\[
u_i\in C^{2,\delta_i}(\overline\Omega),\qquad
\delta_i:=\min\!\left\{\sigma,
\min_{j:\,a_{ij}>0}\alpha_j\right\}.
\]
\end{theorem}

\begin{theorem}\label{thm:constant-rank-intro}
Under the hypotheses of Theorem \ref{thm:existence-power},
\[
\operatorname{rank}(-D^2 U_i)
\]
is constant in $\Omega$ for every $i$.
\end{theorem}

\begin{theorem}\label{thm:strict-intro}
Assume the hypotheses of Theorem \ref{thm:existence-power}. If, in addition,
$\partial\Omega\in C^{2,\sigma}$ for some $0<\sigma<1$ and $\Omega$ is
uniformly convex, then $D^2U_i$ is negative definite in $\Omega$ for every $i$.
In particular, every $U_i$ is strictly concave in $\Omega$.
\end{theorem}

The paper is organized as follows. Section 2 develops graph-compatible powers,
derives the transformed system, establishes the inverse-convexity property,
and presents representative graph examples. Section 3 proves Theorem
\ref{thm:existence-power}, including existence, uniqueness, and power
concavity. Section 4 proves the constant-rank theorem and then derives strict
power concavity.

\section*{Acknowledgments}
The authors thank Professor Xi-Nan Ma for bringing this question to their attention.
Both authors were supported by the National Key Research and Development Program of China
(Grant No. 2025YFA1017601).

\section{Graph-Compatible Powers and Structural Convexity}

\subsection{Graph-compatible powers and the transformed system}

\begin{assumption}\label{ass:structural}
There exist $p_i>0$ such that, for every edge $i\to j$,
\[
\frac1{p_i}=2+\frac{\alpha_j}{p_j},\qquad
0<\frac{\alpha_j}{p_j}\leq1.
\]
\end{assumption}

Since every row has at least one outgoing edge, Assumption \ref{ass:structural} gives
$1/p_i>2$, and hence
\begin{equation}\label{eq:pi-half}
0<p_i<\frac12.
\end{equation}
For a fixed row $i$, all outgoing edges have the same value of $\alpha_j/p_j$, because
\[
\frac{\alpha_j}{p_j}=\frac1{p_i}-2.
\]
We denote this common value by
\[
s_i:=\frac1{p_i}-2.
\]
Thus
\begin{equation}\label{eq:si-range}
0<s_i\leq1.
\end{equation}
Moreover, for each active edge $i\to j$,
\[
\alpha_j=s_ip_j\leq p_j<1.
\]
Thus all active exponents are sublinear.

Let $U_i=u_i^{p_i}$. A direct computation gives
\[
\nabla u_i=\frac1{p_i}U_i^{1/p_i-1}\nabla U_i,
\]
and
\[
\Delta u_i=
\frac1{p_i}U_i^{1/p_i-1}\Delta U_i
+\frac1{p_i}\left(\frac1{p_i}-1\right)U_i^{1/p_i-2}|\nabla U_i|^2.
\]
Substitution into \eqref{eq:system} and multiplication by $p_iU_i^{1-1/p_i}$ yield
\[
-\Delta U_i=
\frac{1-p_i}{p_i}\frac{|\nabla U_i|^2}{U_i}
+p_i\sum_{j:i\to j}a_{ij}U_i^{1-1/p_i}U_j^{\alpha_j/p_j}.
\]
By graph compatibility,
\[
1-\frac1{p_i}=-1-s_i,
\]
and therefore
\begin{equation}\label{eq:transformed-system}
-\Delta U_i=F_i(U,\nabla U_i),
\end{equation}
where
\begin{equation}\label{eq:F-def}
F_i(U,\xi)=U_i^{-1}\left[
\frac{1-p_i}{p_i}|\xi|^2
+p_i\sum_{j:i\to j}a_{ij}\left(\frac{U_j}{U_i}\right)^{s_i}
\right].
\end{equation}

\subsection{Inverse convexity and examples}

We use the following two elementary convexity facts. If $h$ is positive and
concave on a convex set, then $1/h$ is convex. Moreover, if
$h:(0,\infty)^k\to\mathbb R$ is convex, then its perspective
\[
(x,y_1,\ldots,y_k)\mapsto xh(y_1/x,\ldots,y_k/x)
\]
is convex on $x>0$, $y_j>0$; see \cite{Rockafellar1970}.

\begin{proposition}\label{prop:minus-one-concavity}
Under Assumption \ref{ass:structural}, for every fixed $\xi\in\mathbb R^n$, the function $U\mapsto F_i(U,\xi)$ is $(-1)$-concave on the positive cone, namely
\[
U\mapsto \frac1{F_i(U,\xi)}
\]
is convex. Moreover, $F_i$ is nondecreasing in every active component $U_j$, $j\neq i$.
\end{proposition}

\begin{proof}
For fixed $i$ and $\xi$, write
\[
F_i(U,\xi)=U_i^{-1}\left[
A_i+B_{ii}+\sum_{\substack{j:i\to j\\j\neq i}}B_{ij}\left(\frac{U_j}{U_i}\right)^{s_i}
\right],
\]
where
\[
A_i=\frac{1-p_i}{p_i}|\xi|^2\geq0,\qquad B_{ij}=p_ia_{ij}\geq0.
\]
If there is no active edge $i\to j$ with $j\neq i$, then
$1/F_i(U,\xi)=U_i/(A_i+B_{ii})$ is linear in $U_i$, and the conclusion is
immediate. Otherwise, set
\[
g_i(t)=A_i+B_{ii}+\sum_{\substack{j:i\to j\\j\neq i}}B_{ij}t_j^{s_i},\qquad t_j>0.
\]
Because $0<s_i\leq1$, the function $g_i$ is positive and concave. Thus
$h_i=1/g_i$ is convex. Hence
\[
\frac1{F_i(U,\xi)}
=U_ih_i\left(\left(\frac{U_j}{U_i}\right)_{\substack{j:i\to j\\j\neq i}}\right)
\]
is the perspective of a convex function and is therefore convex. The
monotonicity of $F_i$ in each active $U_j$, $j\neq i$, is immediate from
\eqref{eq:F-def}.
\end{proof}

\paragraph{Single-cycle coupling.}

Consider an $m$-cycle
\[
1\to2,\quad 2\to3,\quad \ldots,\quad m\to1.
\]
The system is
\[
-\Delta u_i=\lambda_i u_{i+1}^{\alpha_{i+1}},
\qquad \lambda_i>0,
\]
with indices understood modulo $m$. The compatibility equations are
\[
\frac1{p_i}=2+\frac{\alpha_{i+1}}{p_{i+1}}.
\]
Iterating around the cycle gives
\[
\frac1{p_i}=
\frac{
2(1+\alpha_{i+1}+\alpha_{i+1}\alpha_{i+2}
+\cdots+\alpha_{i+1}\cdots\alpha_{i-1})
}{
1-\alpha_1\alpha_2\cdots\alpha_m
}.
\]
Thus these compatibility equations admit positive powers $p_i$ if and only if
\[
\prod_{i=1}^m\alpha_i<1.
\]
The additional condition \eqref{eq:small-coupling} is
\[
\frac{\alpha_i}{p_i}\leq1
\]
for every $i$.

\paragraph{Fully coupled systems with equal exponents.}

Assume that the graph is fully coupled, by which we mean that $a_{ij}>0$ for all $1\leq i,j\leq m$, and
\[
\alpha_1=\cdots=\alpha_m=\alpha.
\]
Then graph compatibility forces
\[
p_1=\cdots=p_m=:p.
\]
The relation is
\[
\frac1p=2+\frac\alpha p,
\]
so
\[
p=\frac{1-\alpha}{2}.
\]
Thus the compatibility equation \eqref{eq:compatibility} admits a positive solution precisely when $0<\alpha<1$. The additional condition \eqref{eq:small-coupling} is
\[
\frac\alpha p\leq1,
\]
which is equivalent to
\[
\alpha\leq\frac13.
\]
Hence, in the fully coupled equal-exponent case, Theorems
\ref{thm:existence-power}--\ref{thm:strict-intro} apply when
\[
0<\alpha\leq\frac13,\qquad p=\frac{1-\alpha}{2},
\]
provided the domain satisfies the corresponding boundary assumptions.

\section{Proof of Theorem \ref{thm:existence-power}}

This section proves existence, uniqueness, and power concavity. We first
construct the positive solution of \eqref{eq:system} directly, without
introducing an $\varepsilon$-regularized problem.

\subsection{Existence and uniqueness}

Let $\phi_1>0$ be the first Dirichlet eigenfunction,
\[
-\Delta\phi_1=\lambda_1(\Omega)\phi_1,\qquad \phi_1=0\quad\hbox{on }\partial\Omega.
\]
For small $\delta>0$, set
\[
\underline u_i=\delta\phi_1.
\]
For each $i$, choose $j_i$ such that $a_{ij_i}>0$. Since
$0<\alpha_{j_i}<1$,
\[
\delta^{\alpha_{j_i}-1}\longrightarrow+\infty
\qquad\hbox{as }\delta\downarrow0.
\]
Moreover, because $\alpha_{j_i}-1<0$ and $\phi_1$ is bounded above,
\[
\phi_1^{\alpha_{j_i}-1}
\geq \|\phi_1\|_{L^\infty(\Omega)}^{\alpha_{j_i}-1}
\qquad\hbox{in }\Omega.
\]
Since there are only finitely many rows, for all sufficiently small $\delta>0$,
\[
-\Delta\underline u_i
=\delta\lambda_1(\Omega)\phi_1
\leq a_{ij_i}(\delta\phi_1)^{\alpha_{j_i}}
\leq\sum_{j=1}^m a_{ij}(\delta\phi_1)^{\alpha_j}
=\sum_{j=1}^m a_{ij}\underline u_j^{\alpha_j}.
\]
Thus $\underline u$ is a positive subsolution.

Next let $\tau$ be the torsion function
\[
-\Delta\tau=1\quad\hbox{in }\Omega,\qquad \tau=0\quad\hbox{on }\partial\Omega.
\]
Let $T=\|\tau\|_{L^\infty(\Omega)}$. Since all active exponents satisfy $0<\alpha_j<1$, we may choose $M\gg1$ so that, for all $i$,
\[
M\geq\sum_{j=1}^m a_{ij}M^{\alpha_j}T^{\alpha_j}.
\]
Then
\[
\overline u_i=M\tau
\]
is a supersolution because
\[
-\Delta\overline u_i=M
\geq
\sum_j a_{ij}(M\tau)^{\alpha_j}
=\sum_j a_{ij}\overline u_j^{\alpha_j}.
\]
Furthermore, the comparison principle applied to
$\lambda_1(\Omega)\|\phi_1\|_{L^\infty(\Omega)}\tau-\phi_1$ gives
\[
\phi_1\leq\lambda_1(\Omega)\|\phi_1\|_{L^\infty(\Omega)}\tau
\qquad\hbox{in }\Omega.
\]
Taking $\delta$ smaller if necessary, we therefore have
\[
0<\underline u_i\leq\overline u_i\qquad\hbox{in }\Omega.
\]

At this stage, we have a positive subsolution and a torsion supersolution. Next, we use the Schauder fixed point theorem to establish existence.

Let
\[
X=C(\overline\Omega)^m
\]
with the product supremum norm. Define the closed convex set
\[
K=\{w\in X:\underline u_i\leq w_i\leq\overline u_i,\ i=1,\ldots,m\}.
\]
For $w\in K$, define $\mathcal T(w)=z$, where $z_i$ solves
\[
-\Delta z_i=\sum_{j=1}^m a_{ij}w_j^{\alpha_j},\qquad z_i=0\quad\hbox{on }\partial\Omega.
\]
This is a linear Poisson problem. By the monotonicity of the nonlinearities, the comparison principle for the scalar Laplacian, and the subsolution/supersolution inequalities,
\[
\underline u_i\leq z_i\leq\overline u_i.
\]
Hence $\mathcal T(K)\subset K$.

The map $\mathcal T$ is continuous in $C(\overline\Omega)^m$. Indeed, if
$w^k\to w$ uniformly and $z^k=\mathcal T(w^k)$, $z=\mathcal T(w)$, then the
right-hand sides converge uniformly, and comparison with the torsion function gives
\[
\|z_i^k-z_i\|_{L^\infty(\Omega)}
\leq \|\tau\|_{L^\infty(\Omega)}
\left\|\sum_{j=1}^m a_{ij}
\bigl((w_j^k)^{\alpha_j}-w_j^{\alpha_j}\bigr)\right\|_{L^\infty(\Omega)}
\longrightarrow0.
\]
The map is also compact. Indeed, the right-hand sides are uniformly bounded. By the global Calder\'on--Zygmund estimate on $C^{1,1}$ domains \cite{GilbargTrudinger2001}, for every fixed $1<q<\infty$,
\[
\|z_i\|_{W^{2,q}(\Omega)}\leq C_q
\]
uniformly for $w\in K$. Choosing $q>n$ and using Morrey's embedding,
\[
W^{2,q}(\Omega)\hookrightarrow C^{1,\gamma}(\overline\Omega),
\qquad \gamma=1-\frac nq>0,
\]
and then the Arzel\`a--Ascoli theorem, we see that $\mathcal T(K)$ is relatively compact in $C(\overline\Omega)^m$.

By Schauder's fixed point theorem \cite{Amann1976}, $\mathcal T$ has a fixed point $u\in K$. Then $u$ solves \eqref{eq:system}. Since $u_i\geq\underline u_i>0$ in $\Omega$, it is a positive solution. Since $q$ in the estimates above can be chosen arbitrarily large, Morrey's embedding gives
\[
u_i\in C^{1,\gamma}(\overline\Omega)
\qquad\text{for every }0<\gamma<1.
\]
Because $u_j>0$ in $\Omega$, the nonlinearities $u_j^{\alpha_j}$ are
smooth on compact subsets of $\Omega$; interior Schauder bootstrapping then
gives $u_i\in C^\infty(\Omega)$. Finally, if
$\partial\Omega\in C^{2,\sigma}$, the components are Lipschitz on
$\overline\Omega$ and
\[
\sum_{j=1}^m a_{ij}u_j^{\alpha_j}
\in C^{0,\beta_i}(\overline\Omega),\quad
\beta_i:=\min_{j:\,a_{ij}>0}\alpha_j,
\]
since $|s^{\alpha_j}-t^{\alpha_j}|\leq|s-t|^{\alpha_j}$ for
$s,t\geq0$ and every active term. The global boundary Schauder estimate
yields $u_i\in C^{2,\delta_i}(\overline\Omega)$ with
$\delta_i=\min\{\sigma,\beta_i\}$.

We next use a weighted comparison argument to establish uniqueness. For every edge $i\to j$, graph compatibility gives
\begin{equation}\label{eq:weight-strict}
\frac1{p_i}=2+\frac{\alpha_j}{p_j}>\frac{\alpha_j}{p_j}.
\end{equation}
Before stating the comparison lemma, we fix the boundary convention used below. When $v$ and $w$ are defined only in $\Omega$, we interpret
\[
v\leq w\quad\hbox{on }\partial\Omega
\]
as
\begin{equation}\label{eq:boundary-convention}
\limsup_{\Omega\ni x\to y}\bigl(v(x)-w(x)\bigr)\leq0
\qquad\hbox{for every }y\in\partial\Omega.
\end{equation}
For positive Dirichlet functions this is weaker than controlling the quotient $v/w$ near the boundary. The weighted comparison below is designed precisely to avoid extending the quotient to the boundary by normal derivatives. Instead, it uses the linear boundary estimates $v\leq Cd$ and $w\geq cd$, where $d(x)=\operatorname{dist}(x,\partial\Omega)$, and rules out a boundary maximizing sequence by a viscosity Hopf barrier. We use the standard viscosity conventions of \cite{CrandallIshiiLions1992}.

\begin{lemma}\label{lem:weighted-comparison}
Under Assumption \ref{ass:structural}, let $v=(v_1,\ldots,v_m)$ be a bounded positive upper semicontinuous viscosity subsolution of \eqref{eq:system} and let $\overline u=(\overline u_1,\ldots,\overline u_m)$ be a positive classical supersolution. Assume:
\begin{itemize}
\item[(a)] there are constants $c_i,C_i>0$ such that
\[
c_id(x)\leq\overline u_i(x)\leq C_id(x)
\]
in a boundary collar;
\item[(b)] there is a constant $C_i'>0$ such that
\[
0\leq v_i(x)\leq C_i'd(x)
\]
in a boundary collar.
\end{itemize}
Then
\[
v_i\leq\overline u_i\qquad\hbox{in }\Omega,\quad i=1,\ldots,m.
\]
\end{lemma}

\begin{proof}
The linear boundary estimates imply that the weighted quotients
\[
Q_i(x)=\left(\frac{v_i(x)}{\overline u_i(x)}\right)^{p_i}
\]
are bounded in $\Omega$. Set
\[
t_0=\max_i\sup_\Omega Q_i<\infty.
\]
Then
\begin{equation}\label{eq:weighted-upper}
v_i\leq t_0^{1/p_i}\overline u_i\qquad\hbox{in }\Omega
\end{equation}
for every $i$. If $t_0\leq1$, the conclusion follows. Suppose, toward a contradiction, that $t_0>1$.

We first exclude an interior contact. If, for some component $i_0$, equality in \eqref{eq:weighted-upper} occurs at an interior point $x_0$, then
\[
\varphi(x)=t_0^{1/p_{i_0}}\overline u_{i_0}(x)
\]
touches $v_{i_0}$ from above at $x_0$. The viscosity subsolution property gives
\[
-\Delta\varphi(x_0)\leq\sum_j a_{i_0j}v_j(x_0)^{\alpha_j}.
\]
Since $\overline u$ is a supersolution,
\[
-\Delta\varphi(x_0)
=t_0^{1/p_{i_0}}\bigl(-\Delta\overline u_{i_0}\bigr)(x_0)
\geq
t_0^{1/p_{i_0}}\sum_j a_{i_0j}\overline u_j(x_0)^{\alpha_j}.
\]
Using \eqref{eq:weighted-upper},
\[
v_j(x_0)^{\alpha_j}\leq t_0^{\alpha_j/p_j}\overline u_j(x_0)^{\alpha_j}.
\]
Therefore
\[
\sum_j a_{i_0j}\bigl(t_0^{1/p_{i_0}}-t_0^{\alpha_j/p_j}\bigr)
\overline u_j(x_0)^{\alpha_j}\leq0,
\]
which is impossible because $t_0>1$, $\frac1{p_{i_0}}>\frac{\alpha_j}{p_j}$ on every active edge $i_0\to j$, each row has an active edge, and $\overline u_j>0$ in $\Omega$.

It remains to rule out the possibility that the supremum is approached only at the boundary. Since each $Q_i$ is upper semicontinuous in $\Omega$, a maximizing sequence contained in a compact subset of $\Omega$ would yield an interior contact, which has already been excluded. Thus, after passing to a subsequence, there are a component $i_0$, points $x_k\in\Omega$, and $y_0\in\partial\Omega$ such that
\[
x_k\to y_0,\qquad Q_{i_0}(x_k)\to t_0.
\]
Set
\[
w=v_{i_0}-t_0^{1/p_{i_0}}\overline u_{i_0}.
\]
Then $w\leq0$ in $\Omega$. Since
\[
Q_{i_0}(x_k)^{1/p_{i_0}}
=\frac{v_{i_0}(x_k)}{\overline u_{i_0}(x_k)}
\to t_0^{1/p_{i_0}},
\]
we have
\[
\frac{w(x_k)}{d(x_k)}
=
\left(Q_{i_0}(x_k)^{1/p_{i_0}}-t_0^{1/p_{i_0}}\right)
\frac{\overline u_{i_0}(x_k)}{d(x_k)}.
\]
The first factor tends to $0$, while the second factor is bounded by the boundary estimate $\overline u_{i_0}\leq C_id$. Hence
\begin{equation}\label{eq:boundary-slope-zero}
\lim_{k\to\infty}\frac{w(x_k)}{d(x_k)}=0.
\end{equation}
This is the precise meaning of the boundary maximizing sequence: after subtracting the critical multiple $t_0^{1/p_{i_0}}\overline u_{i_0}$, the remaining function has zero first-order boundary slope along the sequence $x_k$.

We claim that $\Delta w\geq0$ in the viscosity sense in $\Omega$. If $\psi$ touches $w$ from above at an interior point, then
\[
\psi+t_0^{1/p_{i_0}}\overline u_{i_0}
\]
touches $v_{i_0}$ from above. Therefore
\[
-\Delta\psi
\leq
\sum_j a_{i_0j}\left(v_j^{\alpha_j}-t_0^{1/p_{i_0}}\overline u_j^{\alpha_j}\right)
\leq
\sum_j a_{i_0j}\left(t_0^{\alpha_j/p_j}-t_0^{1/p_{i_0}}\right)
\overline u_j^{\alpha_j}
\leq0.
\]
Thus $\Delta w\geq0$ in the viscosity sense. We now record the two boundary facts needed to rule out a maximizing sequence approaching $\partial\Omega$.

First, $w$ is not identically zero in any boundary cap centered at $y_0$. Indeed, suppose by contradiction that for some $r>0$,
\[
w\equiv0\qquad\hbox{in }D_r:=\Omega\cap B_r(y_0).
\]
Then
\[
v_{i_0}=t_0^{1/p_{i_0}}\overline u_{i_0}\qquad\hbox{in }D_r.
\]
In the viscosity sense, the function $t_0^{1/p_{i_0}}\overline u_{i_0}$ touches $v_{i_0}$ from above at every interior point of $D_r$. Hence the subsolution inequality and the supersolution inequality give
\[
t_0^{1/p_{i_0}}\sum_j a_{i_0j}\overline u_j^{\alpha_j}
\leq
\sum_j a_{i_0j}v_j^{\alpha_j}.
\]
Using \eqref{eq:weighted-upper},
\[
\sum_j a_{i_0j}v_j^{\alpha_j}
\leq
\sum_j a_{i_0j}t_0^{\alpha_j/p_j}\overline u_j^{\alpha_j}.
\]
Consequently,
\[
\sum_j a_{i_0j}\bigl(t_0^{1/p_{i_0}}-t_0^{\alpha_j/p_j}\bigr)
\overline u_j^{\alpha_j}\leq0
\qquad\hbox{in }D_r.
\]
This is impossible, because every row has at least one active edge, $\overline u_j>0$ in $\Omega$, $t_0>1$, and $\frac1{p_{i_0}}>\frac{\alpha_j}{p_j}$ on every active edge $i_0\to j$. Thus $w$ is not identically zero in any boundary cap centered at $y_0$.

We now apply a uniform Hopf barrier directly to the maximizing sequence. Since $w$ is upper semicontinuous, $w\leq0$, and $w$ is not identically zero, the viscosity strong maximum principle gives
\[
w<0\qquad\hbox{in }\Omega.
\]
Because $\partial\Omega$ is $C^{1,1}$, there is a uniform tubular neighborhood and a radius $R>0$ for the interior ball condition. For all sufficiently large $k$, let $y_k\in\partial\Omega$ be the nearest boundary point to $x_k$. Then
\[
x_k=y_k-d(x_k)\nu(y_k),
\]
where $\nu$ is the exterior unit normal. Set
\[
z_k=y_k-R\nu(y_k),\qquad B_R(z_k)\subset\Omega.
\]
Fix $0<\rho<R/2$ and set
\[
K_\rho:=\overline{\bigcup_{k\geq k_0}\partial B_{R-\rho}(z_k)}
\]
for $k_0$ sufficiently large. Since $y_k\to y_0$, the set $K_\rho$ is compact. Moreover, if
$x\in\partial B_{R-\rho}(z_k)$, then
$B_\rho(x)\subset B_R(z_k)\subset\Omega$, and hence
$d(x,\partial\Omega)\geq\rho$. Thus
$K_\rho\subset\{x\in\Omega:d(x,\partial\Omega)\geq\rho\}$. The upper
semicontinuity and strict negativity of $w$ therefore give a constant
$\eta>0$, independent of $k$, such that
\[
w\leq-\eta\qquad\hbox{on }\partial B_{R-\rho}(z_k)
\]
for all $k\geq k_0$.

In the annulus
\[
A_{k,\rho}=B_R(z_k)\setminus\overline{B_{R-\rho}(z_k)},
\]
set
\[
\beta_k(x)=e^{-\lambda|x-z_k|^2}-e^{-\lambda R^2}.
\]
Choose $\lambda>n/(2(R-\rho)^2)$. Then $\beta_k=0$ on $\partial B_R(z_k)$, $\beta_k>0$ in $A_{k,\rho}$, and
\[
\Delta\beta_k
=
\left(4\lambda^2|x-z_k|^2-2n\lambda\right)e^{-\lambda|x-z_k|^2}>0
\qquad\hbox{in }A_{k,\rho}.
\]
The value of $\beta_k$ on the inner spherical boundary is independent of $k$. Hence one can choose $\varepsilon>0$, also independent of $k$, such that $w+\varepsilon\beta_k\leq0$ on the inner spherical boundary. On the outer spherical boundary, the same inequality holds pointwise at points in $\Omega$ and in the limsup sense \eqref{eq:boundary-convention} at points in $\partial\Omega$, because the boundary estimates imply that both $v_{i_0}$ and $\overline u_{i_0}$ tend to zero there. Thus
\[
w+\varepsilon\beta_k\leq0\qquad\hbox{on }\partial A_{k,\rho}
\]
in this boundary sense.
If $w+\varepsilon\beta_k$ had a positive interior maximum at $x_*$, then
\[
\varphi(x)=-\varepsilon\beta_k(x)+(w+\varepsilon\beta_k)(x_*)
\]
would touch $w$ from above at $x_*$. The viscosity inequality $\Delta w\geq0$ would give
\[
0\leq\Delta\varphi(x_*)=-\varepsilon\Delta\beta_k(x_*)<0,
\]
a contradiction. Therefore $w+\varepsilon\beta_k\leq0$ in $A_{k,\rho}$.

For all sufficiently large $k$, $d(x_k)<\rho$ and
$|x_k-z_k|=R-d(x_k)$. By the mean value theorem, there is a constant
$c_0>0$, independent of $k$, such that
\[
\beta_k(x_k)
=e^{-\lambda(R-d(x_k))^2}-e^{-\lambda R^2}
\geq c_0d(x_k).
\]
Consequently,
\[
\frac{w(x_k)}{d(x_k)}\leq-\varepsilon c_0<0,
\]
which contradicts \eqref{eq:boundary-slope-zero}. Hence the boundary alternative is impossible. Therefore $t_0\leq1$, and the comparison follows.
\end{proof}

\begin{remark}\label{rem:boundary-condition-use}
Lemma \ref{lem:weighted-comparison} does not require one to define the quotient $v_i/\overline u_i$ on $\partial\Omega$. Assumptions (a)--(b) imply that both functions tend to zero at the boundary, and hence the associated boundary comparison holds in the limsup sense \eqref{eq:boundary-convention}. The possible loss of control of the quotient near $\partial\Omega$ is handled by the linear estimates $v_i\leq Cd$ and $\overline u_i\geq cd$. If the weighted quotient tries to attain its maximum at the boundary, the calculation \eqref{eq:boundary-slope-zero} says that the associated difference $w$ has zero distance-normalized slope along a boundary sequence, while the uniform viscosity Hopf barriers force $w(x_k)\leq-cd(x_k)$ along that sequence. This contradiction replaces the classical normal-derivative quotient argument.
\end{remark}

For two positive solutions of \eqref{eq:system}, the boundary linear estimates follow from the standard boundary barrier and Hopf lemma; see \cite{GilbargTrudinger2001,Dalmasso2000,Pao1992}. Applying Lemma \ref{lem:weighted-comparison} in both directions gives uniqueness.

\subsection{Power concavity}

Let $u$ be the unique positive solution constructed above and set
\[
U_i=u_i^{p_i}.
\]
By \eqref{eq:transformed-system},
\[
-\Delta U_i=F_i(U,\nabla U_i).
\]
For a nonnegative function $w\in C(\overline\Omega)$, define its concave
envelope by
\[
w^*(x)=\sup\left\{
\sum_{\ell=0}^n\lambda_\ell w(x_\ell):
x_\ell\in\overline\Omega,\ \lambda_\ell\geq0,\
\sum_{\ell=0}^n\lambda_\ell=1,\
x=\sum_{\ell=0}^n\lambda_\ell x_\ell
\right\}.
\]
It is the least concave majorant of $w$.We record the contact property used
below, which is the elliptic counterpart of the envelope contact arguments in
\cite{Korevaar1983,Kennington1985,Salani2012,IshigeNakagawaSalani2016}.

\begin{lemma}\label{lem:envelope-contact}
Let $\Omega$ be a bounded convex domain with $C^{1,1}$ boundary, and let
$w\in C(\overline\Omega)$ be nonnegative, with $w=0$ on $\partial\Omega$.
If
\[
\frac{w(x)}{d(x,\partial\Omega)}\longrightarrow+\infty
\qquad\text{as }x\to\partial\Omega,
\]
then, for every $x_0\in\Omega$, $w^*(x_0)$ admits an attained
representation whose active points lie in $\Omega$, are contact points of
$w$ and $w^*$, and share a common supporting hyperplane. At differentiable
contact points, their gradients are equal.
\end{lemma}

\begin{proof}
Compactness gives an attained representation
\[
x_0=\sum_\ell\lambda_\ell x_\ell,
\qquad
w^*(x_0)=\sum_\ell\lambda_\ell w(x_\ell).
\]
Let $L$ be a supporting affine function of $w^*$ at $x_0$. Since $w^*$ is
concave and $L\geq w^*$,
\[
w^*(x_0)=\sum_\ell\lambda_\ell w(x_\ell)
\leq\sum_\ell\lambda_\ell w^*(x_\ell)
\leq w^*(x_0)
\]
and
\[
L(x_0)=\sum_\ell\lambda_\ell L(x_\ell)
\geq\sum_\ell\lambda_\ell w^*(x_\ell)=w^*(x_0).
\]
Hence every active $x_\ell$ satisfies
$L(x_\ell)=w^*(x_\ell)=w(x_\ell)$.

If an active point were $y\in\partial\Omega$, then $L(y)=w(y)=0$. For
small $t>0$, $y-t\nu(y)\in\Omega$ and
$d(y-t\nu(y),\partial\Omega)=t$. Thus
$L(y-t\nu(y))=O(t)$, while
$w(y-t\nu(y))/t\to+\infty$, contradicting $L\geq w^*\geq w$.
Therefore all active points are interior. The gradient assertion follows
because $L$ touches $w$ from above at each differentiable contact point.
\end{proof}

Let $\Gamma_i$ be the concave envelope of $U_i$. Hopf's lemma and a boundary
barrier give
\[
c_i d(x)\leq u_i(x)\leq C_i d(x)
\]
in a boundary collar. Since $0<p_i<1$,
\[
\frac{U_i(x)}{d(x)}\longrightarrow+\infty
\qquad\text{as }x\to\partial\Omega,
\]
so Lemma \ref{lem:envelope-contact} applies.

Fix $i$ and $x_0\in\Omega$. After discarding zero-weight terms, choose
$N\leq n$, $x_\ell\in\Omega$, and $\lambda_\ell>0$ such that
\[
x_0=\sum_{\ell=0}^N\lambda_\ell x_\ell,
\qquad
\Gamma_i(x_0)=\sum_{\ell=0}^N\lambda_\ell U_i(x_\ell),
\qquad
\sum_{\ell=0}^N\lambda_\ell=1.
\]
At the contact points,
\begin{equation}\label{eq:common-gradient}
\nabla U_i(x_\ell)=\theta,
\qquad \ell=0,\ldots,N.
\end{equation}
Set
\[
F_\ell=F_i(U(x_\ell),\theta),
\qquad
a_\ell=\frac{F_\ell^{-1}}
{\sum_{r=0}^N\lambda_rF_r^{-1}}.
\]
Then $F_\ell>0$ and $\sum_{\ell=0}^N\lambda_\ell a_\ell=1$. Define near
$x_0$
\[
\varphi(x)=\sum_{\ell=0}^N\lambda_\ell
U_i\bigl(x_\ell+a_\ell(x-x_0)\bigr).
\]
Since
\[
\sum_{\ell=0}^N\lambda_\ell
\bigl(x_\ell+a_\ell(x-x_0)\bigr)=x,
\]
we have $\varphi\leq\Gamma_i$ and
$\varphi(x_0)=\Gamma_i(x_0)$. Moreover,
\[
\nabla\varphi(x_0)=\theta,
\qquad
-\Delta\varphi(x_0)
=\left(\sum_{\ell=0}^N\lambda_\ell F_\ell^{-1}\right)^{-1}.
\]
By Proposition \ref{prop:minus-one-concavity},
\[
-\Delta\varphi(x_0)
\leq
F_i\left(\sum_{\ell=0}^N\lambda_\ell U(x_\ell),\theta\right).
\]
Since
\[
\Gamma_j(x_0)\geq\sum_{\ell=0}^N\lambda_\ell U_j(x_\ell)
\qquad (j\neq i),
\]
and $F_i$ is nondecreasing in these components,
\begin{equation}\label{eq:lower-test-envelope}
-\Delta\varphi(x_0)
\leq F_i(\Gamma(x_0),\nabla\varphi(x_0)).
\end{equation}
If $\psi\in C^2$ touches $\Gamma_i$ from above at $x_0$, then
$\varphi\leq\Gamma_i\leq\psi$ with equality at $x_0$. Hence
\[
\nabla\psi(x_0)=\nabla\varphi(x_0),
\qquad
D^2\psi(x_0)\geq D^2\varphi(x_0),
\]
and \eqref{eq:lower-test-envelope} gives
\[
-\Delta\psi(x_0)
\leq F_i(\Gamma(x_0),\nabla\psi(x_0)).
\]
Thus $\Gamma$ is a viscosity subsolution of the transformed system.

\begin{lemma}\label{lem:monotone-change}
Let $\Gamma=(\Gamma_1,\ldots,\Gamma_m)$, with every $\Gamma_i>0$, be a
viscosity subsolution of
\[
-\Delta\Gamma_i\leq F_i(\Gamma,\nabla\Gamma_i).
\]
Define
\[
\widehat u_i=\Gamma_i^{1/p_i},\qquad i=1,\ldots,m.
\]
Then $\widehat u=(\widehat u_1,\ldots,\widehat u_m)$ is a viscosity subsolution of \eqref{eq:system}.
\end{lemma}

\begin{proof}
Let $\psi\in C^2$ touch $\widehat u_i$ from above at $x_0$. Locally
$\psi>0$, and $\psi^{p_i}$ touches $\Gamma_i$ from above at $x_0$. At this
point,
\[
-\Delta(\psi^{p_i})
=-p_i\psi^{p_i-1}\Delta\psi
+p_i(1-p_i)\psi^{p_i-2}|\nabla\psi|^2,
\]
while $s_ip_j=\alpha_j$ and $p_i(1+s_i)=1-p_i$ give
\[
F_i\bigl(\Gamma,\nabla(\psi^{p_i})\bigr)
=p_i(1-p_i)\psi^{p_i-2}|\nabla\psi|^2
+p_i\psi^{p_i-1}\sum_j a_{ij}\widehat u_j^{\alpha_j}.
\]
The viscosity inequality for $\Gamma_i$ therefore yields
\[
-\Delta\psi(x_0)
\leq\sum_j a_{ij}\widehat u_j(x_0)^{\alpha_j}.
\]
\end{proof}

Set $\widehat u_i=\Gamma_i^{1/p_i}$. Since
$0\leq\Gamma_i\leq\|U_i\|_{L^\infty(\Omega)}$ and $\Gamma_i$ is concave,
$\widehat u$ is bounded and continuous in $\Omega$. By Lemma
\ref{lem:monotone-change}, it is a viscosity subsolution of
\eqref{eq:system}, and $\Gamma_i\geq U_i$ gives
\[
\widehat u_i\geq u_i.
\]

It remains to verify the boundary estimates in Lemma
\ref{lem:weighted-comparison}. Enlarging $C_i$ if necessary, the estimate
$u_i\leq C_i d$ holds throughout $\Omega$. Hence
\[
U_i\leq C_i^{p_i}d^{p_i}.
\]
The function $d^{p_i}$ is concave because $\Omega$ is convex and
$0<p_i<1$. By the minimality of the concave envelope,
\[
\Gamma_i\leq C_i^{p_i}d^{p_i},
\qquad
\widehat u_i\leq C_i d.
\]
Together with the lower estimate $u_i\geq c_i d$ near the boundary, Lemma
\ref{lem:weighted-comparison} gives $\widehat u_i\leq u_i$ in $\Omega$.
Thus $\widehat u_i=u_i$ and $\Gamma_i=U_i$. Therefore $U_i$ is concave,
and Theorem \ref{thm:existence-power} follows.

\section{Constant Rank and Strict Power Concavity}

\subsection{The constant-rank theorem}

We first recall the two convexity tools used in the microscopic argument.

\begin{lemma}\label{lem:all}
The function
\[
A\mapsto \frac1{\operatorname{tr}(A^{-1})}
\]
is concave on the cone $S_{++}^n$ of positive definite symmetric matrices.
\end{lemma}

This is the matrix lemma used in \cite{AlvarezLasryLions1997}.

We recall the level-set form of the Bian--Guan microscopic convexity
principle; see \cite[Theorem~1.2]{BianGuan2010}.

\begin{theorem}\label{thm:bian-guan}
Let $W\in C^{2,1}(\Omega)$ be a convex solution of
\[
F(D^2W,\nabla W,W,x)=0
\]
in a connected domain $\Omega\subset\mathbb R^n$, where
$F\in C^{2,1}(S^n\times\mathbb R^n\times\mathbb R\times\Omega)$. Suppose that:
\begin{itemize}
\item[(i)] the equation is elliptic along $W$, namely $(F^{ij})>0$;
\item[(ii)] $F(0,\nabla W(x),W(x),x)\neq0$ for every $x\in\Omega$;
\item[(iii)] for every $x\in\Omega$, with $p=\nabla W(x)$, the zero sublevel set
\[
\Gamma_F(p):=
\{(A,z,y)\in S_{++}^n\times\mathbb R\times\Omega:
F(A^{-1},p,z,y)\leq0\}
\]
is locally convex at every point of the form $(A,W(x),x)$ belonging to
$\Gamma_F(p)$.
\end{itemize}
Then $\operatorname{rank}D^2W$ is constant in $\Omega$.
\end{theorem}

The earlier, stronger function-convexity condition appears in
\cite[Theorem~1.1]{BianGuan2009}.

\begin{theorem}\label{thm:constant-rank}
Assume the hypotheses of Theorem \ref{thm:existence-power}. Then
$\operatorname{rank}(-D^2U_i)$ is constant in $\Omega$ for every $i$.
\end{theorem}

\begin{proof}

Fix $i$ and set
\[
W_i=-U_i.
\]
Then $W_i\in C^\infty(\Omega)$, $W_i$ is convex, and $D^2W_i=-D^2U_i$ by Theorem \ref{thm:existence-power}. To keep the dependence on the remaining components explicit, we introduce the following notation. For $z<0$ and $x\in \Omega$, define
\[
V_k^{(i)}(z,x)=
\begin{cases}
-z, & k=i,\\
U_k(x), & k\neq i.
\end{cases}
\]
Thus $V^{(i)}(W_i(x),x)=U(x)$. Since $\nabla U_i=-\nabla W_i$, the $i$-th transformed equation can be written as
\[
\Delta W_i=F_i\bigl(V^{(i)}(W_i(x),x),-\nabla W_i\bigr).
\]
Equivalently,
\[
\mathcal F_i(D^2W_i,\nabla W_i,W_i,x)=0,
\]
where
\[
\mathcal F_i(R,p,z,x)
=
\operatorname{tr}R
-F_i\bigl(V^{(i)}(z,x),-p\bigr).
\]
In this notation, the $i$-th component is represented by the scalar value variable $z=W_i$, while the remaining components enter through the explicit $x$-dependence $U_k(x)$, $k\neq i$. The operator $\mathcal F_i$ is smooth on
$S^n\times\mathbb R^n\times(-\infty,0)\times\Omega$.

Theorem \ref{thm:bian-guan} is stated for operators defined for every real
value variable. We therefore use a standard localization. Fix a connected
subdomain $\Omega'\subset\Omega$. Choose an interval
$J\subset(-\infty,0)$ containing $W_i(\overline{\Omega'})$ and a smooth map
$\chi:\mathbb R\to(-\infty,0)$ which equals the identity on a neighborhood
of $\overline J$. Then
\[
\widetilde{\mathcal F}_i(R,p,z,x)
:=\mathcal F_i(R,p,\chi(z),x),
\qquad x\in\Omega',
\]
belongs to
$C^\infty(S^n\times\mathbb R^n\times\mathbb R\times\Omega')$ and agrees
with $\mathcal F_i$ in a neighborhood of the range of
$(D^2W_i,\nabla W_i,W_i,x)$ over $\Omega'$. Since $\chi$ is the identity in a neighborhood of $W_i(\overline{\Omega'})$, the operators $\widetilde{\mathcal F}_i$ and $\mathcal F_i$, together with their local zero sublevel sets, agree near every point relevant to the solution. Hence the structural condition may be checked using $\mathcal F_i$ on $z<0$.

Ellipticity is immediate because
\[
\frac{\partial \mathcal F_i}{\partial R_{ab}}=\delta_{ab}.
\]
Moreover,
\[
\mathcal F_i(0,p,z,x)
=
-F_i\bigl(V^{(i)}(z,x),-p\bigr)<0
\]
for every $z<0$ and $x\in\Omega$, since all components of
$V^{(i)}(z,x)$ are positive.

It remains to verify the Bian--Guan level-set convexity condition. Using the explicit expression \eqref{eq:F-def},
\[
F_i(U,\xi)=U_i^{-1}\left[
\frac{1-p_i}{p_i}|\xi|^2
+p_i\sum_{j:i\to j}a_{ij}\left(\frac{U_j}{U_i}\right)^{s_i}
\right].
\]
Substituting $U=V^{(i)}(z,x)$ and $\xi=-p$ gives
\[
F_i\bigl(V^{(i)}(z,x),-p\bigr)
=
\frac1{-z}\left[
\frac{1-p_i}{p_i}|p|^2+p_ia_{ii}
+p_i\sum_{\substack{j:i\to j\\j\neq i}}a_{ij}
\left(\frac{U_j(x)}{-z}\right)^{s_i}
\right].
\]
Here the self-coupling term is $p_ia_{ii}$ because
$V_i^{(i)}(z,x)/(-z)=1$. Therefore, for fixed $p$,
\[
\mathcal F_i(A^{-1},p,z,x)\leq0
\]
is equivalent to
\[
\operatorname{tr}(A^{-1})
\leq
\frac1{-z}\left[
\frac{1-p_i}{p_i}|p|^2+p_ia_{ii}
+p_i\sum_{\substack{j:i\to j\\j\neq i}}a_{ij}
\left(\frac{U_j(x)}{-z}\right)^{s_i}
\right],
\]
or equivalently
\begin{equation}\label{eq:level-set-expanded}
\frac1{\operatorname{tr}(A^{-1})}
\geq
\frac{-z}{
\displaystyle
\frac{1-p_i}{p_i}|p|^2+p_ia_{ii}
+p_i\sum_{\substack{j:i\to j\\j\neq i}}a_{ij}
\left(\frac{U_j(x)}{-z}\right)^{s_i}}
=
\frac1{F_i\bigl(V^{(i)}(z,x),-p\bigr)}.
\end{equation}
This is the expanded form of the level-set inequality. In particular, every active component $U_j(x)$ with $j\neq i$ remains present, while the self-coupling term is constant inside the brackets.

We now prove that the right-hand side of \eqref{eq:level-set-expanded} is convex in $(z,x)$. Let
\[
G_i(U)=\frac1{F_i(U,-p)}.
\]
By Proposition \ref{prop:minus-one-concavity}, $G_i$ is convex in the positive value variables $U$. Also, since $F_i$ is nondecreasing in every active component $U_j$, $j\neq i$, the function $G_i=1/F_i$ is nonincreasing in those components.

Let $0<\theta<1$ and set
\[
(z_\theta,x_\theta)=(1-\theta)(z_1,x_1)+\theta(z_2,x_2).
\]
For the $i$-th component,
\[
-z_\theta=(1-\theta)(-z_1)+\theta(-z_2).
\]
For every active edge $i\to j$ with $j\neq i$, the concavity of $U_j$ gives
\[
U_j(x_\theta)\geq(1-\theta)U_j(x_1)+\theta U_j(x_2).
\]
Components which are inactive in the $i$-th equation do not enter $G_i$. Since $G_i$ is nonincreasing in each active $U_j$ with $j\neq i$, we obtain
\[
G_i\bigl(V^{(i)}(z_\theta,x_\theta)\bigr)
\leq
G_i\bigl((1-\theta)V^{(i)}(z_1,x_1)+\theta V^{(i)}(z_2,x_2)\bigr).
\]
Using the convexity of $G_i$ in the value variables, the right-hand side is bounded by
\[
(1-\theta)G_i\bigl(V^{(i)}(z_1,x_1)\bigr)
+\theta G_i\bigl(V^{(i)}(z_2,x_2)\bigr).
\]
Thus
\[
(z,x)\mapsto \frac1{F_i\bigl(V^{(i)}(z,x),-p\bigr)}
\]
is convex.

On the other hand,
\[
A\mapsto\frac1{\operatorname{tr}(A^{-1})}
\]
is concave on $S_{++}^n$ by Lemma \ref{lem:all}. Hence the set defined by \eqref{eq:level-set-expanded} is the superlevel set
\[
\left\{
(A,z,x):
\frac1{\operatorname{tr}(A^{-1})}
-\frac1{F_i\bigl(V^{(i)}(z,x),-p\bigr)}
\geq0
\right\}
\]
of a concave function, and is therefore convex. Hence the local
zero-sublevel convexity condition in Theorem \ref{thm:bian-guan} holds at all
relevant points. Applying that theorem to
$\widetilde{\mathcal F}_i$ on $\Omega'$ shows that the rank of
$D^2W_i=-D^2U_i$ is constant in $\Omega'$. Since any two points of the
convex domain $\Omega$ are contained in some connected subdomain
$\Omega'\Subset\Omega$, the rank is constant throughout $\Omega$.
\end{proof}

\subsection{Boundary strictness and strict power concavity}
We next record the boundary convexity lemma needed below. For completeness, we include its proof. Related boundary arguments appear in Caffarelli and Friedman \cite{CaffarelliFriedman1985} and Korevaar \cite{Korevaar1983}. Throughout this subsection, if $\nu$ denotes the exterior unit normal, we use the convention
\[
\mathrm{II}(\tau,\tau):=\langle D_\tau\nu,\tau\rangle.
\]
Thus uniform convexity means that there exists $\kappa>0$ such that
$\mathrm{II}(\tau,\tau)\geq\kappa|\tau|^2$ for every tangent vector $\tau$ on $\partial\Omega$.

\begin{lemma}\label{lem:boundary1}
Let $\Omega\subset\mathbb R^n$ be a bounded $C^2$,
uniformly convex domain. Let
$u\in C^{\infty}(\Omega)\cap C^{2}(\overline{\Omega})$ satisfy
\begin{equation}\label{eq:boundary-data}
u<0\ \text{in }\Omega,\quad u=0\ \text{on }\partial\Omega,\quad \partial_\nu u>0\ \text{on }\partial\Omega,
\end{equation}
where $\nu$ is the exterior normal to $\partial\Omega$. Let $f\in C^2((-\delta,0))$ for some $\delta>0$, and set $v=f(u)$ in a boundary collar where $-\delta<u<0$. If
\[
f'>0,\qquad f''>0,\qquad \lim_{s\to0^{-}}\frac{f'(s)}{f''(s)}=0,
\]
then there exists $\varepsilon>0$ such that $v$ is strictly convex in
\[
\{x\in\Omega:0<d(x,\partial\Omega)<\varepsilon\}.
\]
\end{lemma}

\begin{proof}
On $\partial\Omega$, $\nabla u=(\partial_\nu u)\nu$. If $\tau$ is tangent to $\partial\Omega$, differentiating the identity $u=0$ twice along the boundary gives
\[
D^2u(\tau,\tau)=\partial_\nu u\,\mathrm{II}(\tau,\tau),
\]
with the convention for $\mathrm{II}$ fixed above. Uniform convexity and the positivity of $\partial_\nu u$ give a uniform
positive lower bound for this quadratic form on unit tangent vectors at the
boundary. By compactness and continuity, the bound persists in a sufficiently
thin boundary strip:
\begin{equation}\label{eq:level-tangent-positive}
D^2u(\eta,\eta)\geq c_1|\eta|^2
\qquad\hbox{whenever }\eta\perp\nabla u,
\end{equation}
for some $c_1>0$. Indeed, otherwise one could take points approaching the
boundary and unit vectors orthogonal to $\nabla u$ for which the estimate
fails; a convergent subsequence would yield a tangent unit vector at the
boundary and contradict the boundary lower bound. After shrinking the strip,
we also have $|\nabla u|\geq c_2>0$ there.

Set $q(u)=f'(u)/f''(u)$. Since $q(u)\to0$ as $u\to0^-$,
\[
\frac{D^2v}{f''(u)}
=
\nabla u\otimes\nabla u+q(u)D^2u.
\]
Let $e=\nabla u/|\nabla u|$ and write $\xi=ae+\eta$, where
$\eta\perp e$. Since $D^2u$ is bounded in the boundary strip,
\eqref{eq:level-tangent-positive} and Young's inequality give a constant $C>0$ such that
\[
\bigl(\nabla u\otimes\nabla u+q(u)D^2u\bigr)[\xi,\xi]
\geq
\bigl(c_2^2-Cq(u)\bigr)a^2+\frac{c_1}{2}q(u)|\eta|^2.
\]
For a sufficiently thin strip, $q(u)$ is small enough that the right-hand side is positive for every nonzero $\xi$. Since $f''(u)>0$, this proves $D^2v>0$ in that strip.
\end{proof}

\begin{remark}
The boundary regularity assertion in Theorem
\ref{thm:existence-power} supplies precisely the
$C^2(\overline\Omega)$ regularity needed here when
$\partial\Omega\in C^{2,\sigma}$. No regularity of $u^p$ up to $\partial\Omega$ is required.
\end{remark}

By Lemma~\ref{lem:boundary1}, we obtain the following conclusion.

\begin{lemma}\label{lem:boundary-strict}
Let $\Omega\subset\mathbb R^n$ be a bounded $C^2$, uniformly convex domain. Let
$u\in C^\infty(\Omega)\cap C^2(\overline\Omega)$ satisfy
\[
u>0\quad\hbox{in }\Omega,\quad
u=0\quad\hbox{on }\partial\Omega,\quad
\partial_\nu u<0\quad\hbox{on }\partial\Omega,
\]
where $\nu$ is the exterior unit normal. If $0<p<1$, then
\[
D^2(u^p)<0
\]
in a sufficiently thin boundary strip.
\end{lemma}

\begin{proof}
Apply Lemma \ref{lem:boundary1} to $w=-u<0$ and
\[
f(s)=-(-s)^p,\quad s<0.
\]
Indeed,
\[
f'(s)=p(-s)^{p-1}>0,\quad
f''(s)=p(1-p)(-s)^{p-2}>0,\quad
\frac{f'(s)}{f''(s)}=\frac{-s}{1-p}\longrightarrow0
\]
as $s\to0^-$. Hence $f(w)=-u^p$ is strictly convex in a sufficiently
thin boundary strip, which is equivalent to $D^2(u^p)<0$ there.
\end{proof}

\begin{theorem}\label{thm:strict}
Assume the hypotheses of Theorem \ref{thm:existence-power}. If, in addition,
$\partial\Omega\in C^{2,\sigma}$ for some $0<\sigma<1$ and $\Omega$ is uniformly convex, then $D^2U_i$ is negative definite in $\Omega$ for every $i$.
In particular, every $U_i$ is strictly concave in $\Omega$.
\end{theorem}

\begin{proof}
Theorem \ref{thm:existence-power} gives
$u_i\in C^2(\overline\Omega)\cap C^\infty(\Omega)$ and $-D^2U_i\geq0$. By Theorem
\ref{thm:constant-rank}, $\operatorname{rank}(-D^2U_i)$ is constant in
$\Omega$. Combing Hopf's lemma with Lemma
\ref{lem:boundary-strict}, we have $-D^2U_i>0$ near $\partial\Omega$. Hence the constant rank is $n$. Therefore $D^2U_i<0$ throughout $\Omega$, and $U_i$ is strictly concave for every $i$.
\end{proof}

\small
\bibliographystyle{alpha}
\bibliography{ref}

\medskip
\noindent
(Jiahuan Li) Department of Mathematics, University of Science and Technology
of China, Hefei 230026, Anhui Province, China.
Email address:
\href{mailto:jiahuan@mail.ustc.edu.cn}
{\texttt{jiahuan@mail.ustc.edu.cn}}.

\medskip
\noindent
(Hanpeng Zhou) Department of Mathematics, University of Science and Technology
of China, Hefei 230026, Anhui Province, China.
Email address:
\href{mailto:zhouhanpeng@mail.ustc.edu.cn}
{\texttt{zhouhanpeng@mail.ustc.edu.cn}}.

\end{document}